\documentclass[12pt]{article}

\usepackage{graphicx}
\usepackage{amsmath}
\usepackage{amssymb}
\usepackage{theorem}

\numberwithin{equation}{section}

\newtheorem{thm}{Theorem}[section]
\newtheorem{lem}[thm]{Lemma}
\newtheorem{pro}[thm]{Proposition}
\newtheorem{cor}[thm]{Corollary}
{\theorembodyfont{\rmfamily}
\newtheorem{defn}[thm]{Definition}
\newtheorem{examp}[thm]{Example}

\newtheorem{rmk}[thm]{Remark}
}

\newcommand{\qed}{\hfill \mbox{\raggedright \rule{.07in}{.1in}}}
 
\newenvironment{proof}{\vspace{1ex}\noindent{\bf
Proof}\hspace{0.5em}}{\hfill\qed\vspace{1ex}}

\newcommand{\Z}{\mathbb{Z}}
\newcommand{\R}{\mathbb{R}}
\newcommand{\N}{\mathbb{N}}

\newcommand{\ep}{\epsilon}

\title{Deterministically Driven Random Walks in a Random Environment on $\Z$.}
\author{Colin Little\thanks{Department of Mathematics, University of Surrey, Guilldford, GU2 7XH, UK.}}

\begin{document}

\maketitle

\begin{abstract}
We introduce the concept of a deterministic walk in a deterministic environment on a countable state space (DWDE). For the deterministic walk in a fixed environment we establish properties analogous to those found in Markov chain theory, but for systems that do \textbf{not} in general have the Markov property. In particular, we establish hypotheses ensuring that a DWDE on $\Z$ is either recurrent or transient. An immediate consequence of this result is that a \emph{symmetric} DWDE on $\Z$ is recurrent. Moreover, in the transient case, we show that the probability that the DWDE diverges to $+ \infty$ is either 0 or 1. In certain cases we compute the direction of divergence in the transient case.
\end{abstract}


\section{Introduction}
Our purpose is to provide a foundational background for the study of deterministically driven random walks in a random environment. Here we consider the simplest infinite state case: deterministically driven random walks in a random environment on $\Z$. 

Much attention in the area of deterministically driven random walks in a random environment has focused lately on applications to problems in randomly generated dynamical billiards. Among the most well-known examples of a dynamical billiard is that of a \emph{Lorentz gas}, which models the free motion of a point particle in Euclidean space $\R^d$ ($d \geq 2$) subject to elastic collisions with a fixed array of dispersing scatterers.

In an early development of the subject of dynamical billiards Sinai \cite{Sinai1970} showed that a finite horizon planar periodic Lorentz gas is recurrent and ergodic. Bunimovich, Sinai and Chernov \cite{BuChSi1981} proved the central limit theorem and weak invariance principle (WIP) for such systems, while Melbourne \& Nicol \cite{MelbNic2007} showed that almost every trajectory is approximable by a sample path of a Brownian motion. (This property is known as an \emph{almost sure invariance principle}.) In a recent development of the subject of deterministic walks in a random environment Dolgopyat et al \cite{DolSzVar2009} have shown that the WIP of a planar periodic Lorentz gas is (in some sense) robust to random perturbations of finite regions of an otherwise periodic arrangement of scatterers. One important feature of the results that we present here is that they do not require an asymptotic periodicity condition on the environment.

Another class of random dynamical billiards considers the situation in which a space (e.g. a strip in $\R^2$, $\R^2$, etc) is tessellated by countably many translated copies of the same polygon. A billiard is then constructed by randomly assigning a configuration of scatterers to each polygon. Once an environment of scatterers has been generated the resulting billiard dynamics are run. Such systems are known as \emph{random Lorentz gases} (or \emph{random Lorentz tubes} in the strip setting). Lenci et al \cite{CriLenSer2010, CriDegEpLenSer2011} show that (under certain hypotheses on the geometry of scatterer configurations) almost all realisations of a random 2-dimensional Lorentz tube are recurrent, and that (under somewhat more restrictive hypotheses) this can also be shown for higher dimensional cases. Lenci \& Troubetzkoy \cite{LenTro2011} identify very special hypotheses under which a given fixed 2-dimensional Lorentz tube is recurrent and ergodic, and for which the first return map is $K$-mixing, and they consider random Lorentz tubes whose typical realisation satisfies these hypotheses.

Examples of other (non-billiard) applications of deterministic walks in random environments can be found in Dolgopyat \cite{Dolgopyat2009}, and Simula \& Stenlund \cite{SimSten2009}.

We observe that a common feature of the above examples is that once an environment (of scatterers) has been generated, it remains unchanged thereafter. A different class of deterministically driven random walk in a random environment deals with situations in which the initial configuration of the environment is randomly chosen, and is then allowed to evolve over time. For example, Stenlund \cite{StenArx2012} considers the asymptotic behaviour of a random Sinai billiard on a 2-torus in which the position of a single scatterer is randomly updated after each collision (for which he establishes an almost sure invariance principle). We do not consider such systems here, but note (as Stenlund does) that the situation in which the environment is frozen is typically much harder to analyse than the situation in which it updates randomly after collisions.

In this work we are concerned with understanding the problem of recurrence and transience properties of a deterministic walk in a random environment on $\Z$. Our primary motivation lies in the fact that (to our knowledge) the existing dynamical systems literature does not address such questions in a systematic way. The secondary motivation consists in the hope that the results described here will provide a foundation upon which to tackle harder problems such as the billiard systems described above.

Our setup bears certain similarities to the now classical subject of \emph{random walks in a random environment} (RWRE), which considers ensembles of randomly generated Markov chains on a common state space. But unlike in that situation we will be concerned more generally with ensembles of dynamical systems that are deterministic in nature, and which do not have the Markov property. 

In Section \ref{sec:DWDEsetup}, we formally introduce our setup. By way of further motivation for the questions considered there, we here give a brief account of RWREs in the one-dimensional setting. 

Consider the following example of a RWRE on $\Z$. Let $(\alpha_n)_{n \in \Z}$ be a sequence of independent and identically distributed (i.i.d.) random variables taking values in $[0,1]$. For each realisation $(\alpha_n)_n$ of this process we define a Markov chain $(U_n)_{n \geq 0}$ on $\Z$ satisfying, (i) $P(U_0 = 0) = 1$ and (ii) for all $n \geq 0$ and $k \in \Z$, $P(U_{n+1} = k+1| U_n = k) = \alpha_k$ and $P(U_{n+1} = k-1| U_n = k) = 1-\alpha_k$. Each realisation $(\alpha_n)_n$ thus defines an \emph{environment} of transition probabilities on $\Z$, and as such the process $(\alpha_n)_{n}$ is called a \emph{random environment} on $\Z$. A setup of this kind is called a \emph{simple} RWRE on $\Z$. More complex setups for the RWRE have focused on situations where the environment is generated by different types of dynamics, and by allowing jumps other than $\pm 1$.

The first mathematical results for simple RWREs go back to the works of Koslov \cite{Koslov1973}, Kesten et al \cite{KeKS1975}, and Solomon \cite{Sol1975} in the mid-1970s. Since then a great deal has been established about the qualitative and quantitative properties of RWREs in the one dimensional setting. (Full and relatively modern surveys of RWREs, including the higher dimensional cases, can be found in \cite{Sznitman2004} and \cite{Zeitouni2001}.)

The seminal results in the area are due to Solomon \cite{Sol1975} and Sinai \cite{Sinai1982}. Solomon was the first to establish a criterion for classifying the (almost sure) pointwise asymptotic behaviour of the simple RWRE on $\Z$. In particular he showed this behaviour was completely determined by the sign of $E(\ln \frac{1-\alpha_0}{\alpha_0})$. Solomon also proved a Law of Large Numbers for simple RWREs, and showed the existence of transient simple RWREs for which the growth rate of the position of the walk is sub-linear. Sinai \cite{Sinai1982} famously showed that in the recurrent case the distributional behaviour of the simple RWRE exhibits anomalously sub-diffusive growth according to a $(\ln n)^2$ law.

Subsequent developments in the field have extended these results to more general settings. Key \cite{Key1984} obtained a more general classification result for the case of RWRE on $\Z$ in an i.i.d.\ environment for which jumps are uniformly bounded. Bolthausen \& Goldsheid \cite{BalGol2000} further extended the classification of RWREs on $\Z$ with bounded jumps to the situation in which the random environment is ergodic and stationary. Moreover, \cite{BalGol2000} considered the more general \emph{strip} model in which the state space is $k \times \Z$, for some $k \in \N$. Bremont \cite{Bremont2004}, under certain additional assumptions, further extended the classification of one dimensional RWREs with bounded jumps to the setting of a Gibbsian environment on a subshift of finite type. As yet, relatively little is known about the classification of higher dimensional RWREs. 

Taking our lead from the probability theory, and looking instead at deterministic dynamical systems, our goal is to establish results of a similar flavour to those obtained for RWREs on $\Z$. Typically these problems are a lot harder in the deterministic setting because very different assumptions regarding the dependence/independence of observables obtain to those that are prevalent in the probabilistic setting.

In preparation for our main results, in Section \ref{Sec:MarkovDistortion} we show that standard properties of Markov chains go over to Markov maps with good distortion properties.

In Section \ref{sec:DWDEsetup}, we formally introduce the notion of central interest: the \emph{deterministic walk in a deterministic environment} (DWDE). Our aim is to provide a foundational understanding of the asymptotic behaviour of a DWDE on $\Z$. In particular, we establish results of the flavour of Solomon \cite{Sol1975} in situations where the deterministic walk typically \textbf{does not} have the Markov property. In the finite state case (which is covered in \cite{Little2012b}) it is possible to prove recurrence under relatively mild assumptions on the underlying dynamics. In the infinite state case it transpires that additional properties such as good distortion and big images are required.

Our first main result (Theorem \ref{thm:Transience_TransitivityofOrbits}) identifies hypotheses that ensure that a deterministic walk on a general countable state space in a given fixed environment is recurrent with probability 0 or 1.

Building upon this, we establish zero-one laws for the asymptotic behaviour of the DWDE on $\Z$. Firstly (in Theorem \ref{thm:4Case0-1Law}), we show that under certain assumptions the DWDE on $\Z$ is recurrent with probability 0 or 1. Secondly (in Theorem \ref{thm:nosplitcase}), we identify conditions that ensure that, with probability 1, the DWDE on $\Z$ exhibits exactly one of three types of asymptotic behaviour: specifically, either (i) the DWDE on $\Z$ diverges to $+\infty$ with probability 1, (ii) the DWDE on $\Z$ diverges to $-\infty$ with probability 1, or (iii) the DWDE is recurrent with probability 1. It follows immediately from Theorem \ref{thm:nosplitcase} that if a DWDE on $\Z$ is \emph{symmetric} (in a sense to be made precise) then it is recurrent (Corollary \ref{cor:symmetricRecurrence}). 

In general, there does not a exist a sharp criterion for determining the transience or recurrence of a DWDE. However, for specific classes of DWDE on $\Z$ for which there is inherent bias in each possible choice of transition function, we make use of Theorem \ref{thm:nosplitcase} to show that such DWDEs are transient, and that they diverge in the expected direction.

\section{Markov Maps with Strong Distortion}\label{Sec:MarkovDistortion}

In this section we recall the concept of a \emph{Markov map}, and the related property of \emph{Strong Distortion}. In preparation for our main results, we establish some useful properties Markov maps with Strong Distortion.

\begin{defn} Recall that a non-singular transformation $T$ of a measure space $(X, m)$ is said to be \emph{Markov}, with a measurable \emph{Markov partition} $\beta$ of $X$, if
\begin{itemize}
\item[(i)] for all $a \in \beta$, $T(a)$ is the union of elements of $\beta$,
\item[(ii)] $T_{|a}: a \rightarrow T(a)$ is a bijection.
\end{itemize}
Given a Markov map $T$ with with Markov partition $\beta$, for all $n \geq 1$ and all sequences $a_0, \ldots, a_{n-1} \in \beta$, we call the set $[a_0, \ldots, a_{n-1}] := \cap_{j=0}^{n-1}T^{-j}a_j$ a \emph{cylinder set of rank} $n$ or an $n$-\emph{cylinder}. 

For a given cylinder set $a$, we let $|a|$ denote its rank. 

We denote by $\beta_n$ the collection of all $n$-cylinders. 

We say that a cylinder set $a$ is \emph{admissible} if $m(a) > 0$.
\end{defn}

In Definition \ref{def:MarkovCommunication} and Definition \ref{def:MarkovRecTrans} below we introduce concepts for Markov maps that are analogous to related concepts for Markov chains. 

\begin{defn}\label{def:MarkovCommunication}
Let $T$ be a Markov transformation of a measure space $(X, m)$ with Markov partition $\beta$. Given $a, b \in \beta$ we say that $a$ \emph{communicates with} $b$ if
\[
m(\{x \in a: T^nx \in b~\mathrm{for~some}~n \geq 1\}) > 0.
\]
Given $a, b \in \beta$, we say that $a$ and $b$ \emph{intercommunicate} if $a$ communicates with $b$ and $b$ communicates with $a$.

Given $a \in \beta$ we define the \emph{communication class of} $a$ to be the set of all $b \in \beta$ with which it intercommunicates. (We observe that intercommunication is an equivalence relation.)

We say that a communication class $\mathcal{C}$ is \emph{closed} if for all $a \in \mathcal{C}$ and all $b \in \beta$, $a$ communicates with $b$, only if $b$ communicates with $a$. (When $\beta$ is finite, the existence of closed communication classes is automatic.)

We say that $\beta$ is \emph{irreducible} if $\beta$ is itself a communication class under $T$.
\end{defn}

\begin{defn}\label{def:MarkovRecTrans}

We say that a partition element $a \in \beta$ is \emph{recurrent} if for a.e.\ $x \in a$ there exists $n \geq 1$ such that $T^nx \in a$. It follows from the non-singularity of $T$ that if $a$ is recurrent then for a.e.\ $x \in a$, $T^nx \in a$ infinitely often. 

We say that $a \in \beta$ is \emph{transient} if it is not recurrent.

We say that a communication class $\mathcal{C}$ is \emph{transitive} if for all $a, b \in \mathcal{C}$, and for a.e.\ $x \in a$, there exists $n \geq 1$ such that $T^nx \in b$. We say that $T$ is \emph{transitive} if the Markov partition $\beta$ forms a transitive communication class.
\end{defn}

\begin{defn}\label{def:StrongDistortion}
Let $T: X \rightarrow X$ be a Markov transformation of a $\sigma$-finite measure space $(X, m)$, with Markov partition $\beta$. 
By the non-singularity of $T$, for all $n \geq 1$ and all $a \in \beta_n$, we may define the Radon-Nikodym derivative $v'_a := \frac{d(m \circ T^{-n})}{dm}$ such that for every measurable set $B$, $\int_{T^na \cap B} v'_a dm = m(a \cap T^{-n}B)$.

We say that $T$ has the \emph{Strong Distortion Property}, with distortion constant $D \geq 1$, if for all $n \geq 1$ and for all $a \in \beta_n$, and for a.e.\ $x, y \in T^n a$, $\frac{v'_a(x)}{v'_a(y)} \leq D$.
\end{defn}

We now give some useful properties of Markov maps with strong distortion. Throughout we assume that $T: X \rightarrow X$ is a Markov transformation of a $\sigma$-finite measure space $(X, m)$ with Markov partition $\beta$, such that $m(a) < \infty$ for all $a \in \beta$, and that $T$ has the Strong Distortion Property. 

The following result derives from \cite[Proposition 4.3.1]{Aaronson1997}.

\begin{pro}\label{pro:GibbsPropimpliesLeakage}
For all $a \in \beta$, $m(Ta) < \infty$ and there exists $C \geq 1$ such that for all $n \geq 1$, $a \in \beta_n$, and any measurable set	$B$
\begin{equation}\label{eqn:Distortion}
C^{-1} \frac{m(B \cap T^na)}{m(T^n a)} \leq \frac{m(a \cap T^{-n}B)}{m(a)} \leq C \frac{m(B \cap T^na)}{m(T^n a)}.
\end{equation}
\end{pro}

\begin{lem}\label{lem:Recimp_ioRec}
If $a \in \beta$ is recurrent, then $T^r x \in a$ $\mathrm{i.o.}$ for a.e.\ $x \in a$.
\end{lem}
\begin{proof}
Let $a' := \{x \in a: T^r x \notin a, \forall r \geq 1\}$. By assumption, $m(a') = 0$. Defining $E := \{x \in a: T^nx \in a~\mathrm{i.o.}\}$, it follows that $E = a \backslash (a \cap \cup_{n=0}^{\infty}T^{-n}a')$. Since $T$ is non-singular it follows that $m(E) = m(a) - m(a \cap \cup_{n=0}^{\infty}T^{-n}a') = m(a)$. 
\end{proof}

\begin{lem}\label{lem:Trans_iozero}
If $a \in \beta$ is transient, then $m(\{x \in X: T^n(x) \in a~\mathrm{i.o.} \}) = 0$.
\end{lem}
\begin{proof}
Let $a' := \{x \in a: T^n x \notin a, \forall n \geq 1\} = a \cap (\bigcap_{n=1}^{\infty} T^{-n}a^c)$. Thus, the set $a'$ contains those points in $a$ whose orbits immediately leave $a$ and never return. Clearly, $a'$ is measurable and, by assumption, $m(a') > 0$. For all $r \geq 1$, we define sets
\[
E_r := \{x \in a: \exists~ 0 < n_1 < \ldots < n_r: T^{n_j}(x) \in a,~\mathrm{for}~ j = 1, \ldots, r\}.
\]
Thus, $E_r$ denotes the set of points $x \in a$ that make at least $r$ return visits to $a$. 

Consider a cylinder $w = [w_0, \ldots, w_{n-1}]$ that contains exactly $r+1$ occurrences of the symbol $a$ and for which $w_0 = w_{n-1} = a$. We define $D_r$ to be the collection of all such cylinders. It follows that $E_r = \cup_{w \in D_r} w$, and that $E_r$ is therefore measurable. Since $Ta' \subset T^{|w|}w$, it follows from Proposition \ref{pro:GibbsPropimpliesLeakage} that there exists $C \geq 1$ such that for each cylinder $w \in D_r$, 
\begin{equation}\label{ineqn:leak1}
\frac{m(w \cap T^{-|w|} (Ta'))}{m(w)} \geq C^{-1} \frac{m(Ta')}{m(T^{|w|}w)} = C^{-1} \frac{m(Ta')}{m(Ta)}.
\end{equation}
We define $\delta := \frac{m(Ta')}{Cm(Ta)}$. From Proposition \ref{pro:GibbsPropimpliesLeakage} we have that $m(Ta) < \infty$, and since $m(a') > 0$, and $T$ is non-singular, it follows that $m(Ta') > 0$, and so $\delta > 0$. From (\ref{ineqn:leak1}) it follows that for all cylinders $w \in D_r$, $m(w \cap (E_{r+1})^c) \geq \delta m(w)$. Taking unions over all cylinders $w \in D_r$, we have
\begin{equation}
m(E_r \backslash E_{r+1}) = m(\bigcup_{w \in D_r} w \cap (E_{r+1})^c) \geq \delta m(\bigcup_{w \in D_r} w) = \delta m(E_r).
\end{equation}
Since $E_{r+1} \subset E_r$, it follows that 
\[
m(E_r) - m(E_{r+1}) \geq \delta m(E_r).
\]
This establishes
\begin{equation}\label{eqn:DecayofCr1}
m(E_{r+1}) \leq (1-\delta)m(E_r).
\end{equation}
For all $r \in 1$, $\{x: T^n(x) \in a~\mathrm{i.o.} \} \subset E_{r+1} \subset E_r$,  and so 
\[
m(\{x \in a: T^n(x) \in a~\mathrm{i.o.} \}) \leq m(E_{r+1}) \leq (1-\delta)m(E_r) \leq (1-\delta)^rm(a).
\]
Thus $m(\{x \in a: T^n(x) \in a~\mathrm{i.o.} \}) = 0$. It follows from the non-singularity of $T$ that for all $b \in \beta$, $m(\{x \in b: T^n(x) \in a~~\mathrm{i.o.}\}) = 0$, and therefore that $m(\{x \in X: T^n(x) \in a~~\mathrm{i.o.}\}) = 0$. 
\end{proof}

\begin{lem}\label{lem:EquivRec}
If $a \in \beta$ is recurrent, then for all $b$ that lie in the same communication class as $a$, $b$ is recurrent.
\end{lem}
\begin{proof}
Suppose that $a$ is recurrent, and fix $b$ such that $a$ and $b$ intercommunicate. Define 
\[
E := \{x \in a: T^n(x) \in a~\mathrm{i.o.} ~\&~ \forall n \geq 1,~ T^n(x) \notin b\}.
\]
We first show that 
\begin{equation}\label{eqn:equivofRecurrence}
m(E) = 0.
\end{equation}
Since $a$ communicates with $b$ there exists a cylinder 
\[
u := [d_1, \ldots , d_k, b]
\]
such that $a \cap T^{-1}u$ is admissible. For $r \geq 1$, define the sets
\begin{eqnarray*}
E_r := \{x \in a: \exists~ 0 < n_1 < \ldots < n_r: T^{n_j}(x) \in a, ~\mathrm{for}~ j = 1, \ldots, r, & ~\&~ & \\ T^s(x) \notin b, ~\mathrm{for}~s = 0, \ldots, n_r)\}.~~~~~~~~~~~~~~~~~~& & \qquad
\end{eqnarray*}
Thus $E_r$ denotes the set of points $x \in a$ that make at least $r$ return visits to $a$ before visiting $b$, if ever. For $r \geq 1$, recall from the proof of Lemma \ref{lem:Trans_iozero} the sets $D_r$ and define the collection of cylinder sets
\[
D_{r,b} := \{w \in D_r: w_j \neq b,~\forall~j = 0, \ldots , |w|-1\}.
\]
Observe that $E_r = \cup_{w \in D_{r,b}} w$. Therefore, $E_r$ is measurable. It follows from Proposition \ref{pro:GibbsPropimpliesLeakage} that there exists $C \geq 1$ such that for each $w \in D_{r,b}$
\begin{equation}\label{eqn:leak1}
\frac{m(w \cap T^{-|w|} u)}{m(w)} \geq C^{-1}\frac{m(u)}{m(Ta)}.
\end{equation}
Defining $\delta := C^{-1}\frac{m(u)}{m(Ta)}$, a similar argument to that employed in Lemma \ref{lem:Trans_iozero} establishes that $m(E_{r+1}) \leq (1-\delta)m(E_r)$. For all $r \geq 1$, $E \subset E_{r+1} \subset E_r$,  and so 
\[
m(E) \leq m(E_{r+1}) \leq (1-\delta)m(E_r) \leq (1-\delta)^rm(a),
\]
from which (\ref{eqn:equivofRecurrence}) now follows. 

Letting $D := \{x \in a: T^n(x) \in a ~\mathrm{i.o.} ~\&~ \exists N \forall n \geq N, T^n(x) \notin b\}$, it follows that $D \subset a \cap (\cup_{n=0}^{\infty}T^{-n}E)$.  Since $T$ is non-singular it follows from (\ref{eqn:equivofRecurrence}) that $m(D) = 0$. Since $a$ is recurrent it follows from Lemma \ref{lem:Recimp_ioRec} that $T^rx \in b$ i.o. for a.e.\ $x \in a$.
 
Since $b$ communicates with $a$, there is a positive measure subset of $b$ whose orbits enter $a$, and therefore return to $b$ infinitely often. By Lemma \ref{lem:Trans_iozero} it is now immediate that $b$ is recurrent.
\end{proof}

From Lemma \ref{lem:EquivRec} we obtain the following result.
\begin{lem}\label{cor:Transience}
Let $\mathcal{C}$ be a communication class. Then either (i) every $b \in \mathcal{C}$ is recurrent, or (ii) every $b \in \mathcal{C}$ is transient.
\end{lem}

In consequence of Lemma \ref{cor:Transience} we may say that a communication class is either \emph{recurrent} or \emph{transient}. From the proof of Lemma \ref{lem:EquivRec}, we have the following.

\begin{lem}\label{cor:Transitivity}
A communication class is recurrent if and only if it is transitive.
\end{lem}

\section{Introducing Deterministic Walks in a Deterministic Environment}\label{sec:DWDEsetup}

We formally define the \emph{deterministic walk} and those of its properties that are of primary interest: specifically, the \emph{transience}, \emph{recurrence} and \emph{transitivity} of the deterministic walk on its state space.

\begin{defn}\label{def:DetWalk}
Let $T$ be a measurable transformation of a probability space $(X, m)$, and let $S$ be a countable set such that associated with each element $i \in S$ is a measurable function $f_i: X \rightarrow S$. We call each $f_i$ a \emph{transition function}, and we call the collection $(f_k)_{k \in S}$ an \emph{environment} on $S$. 

Define the skew-product transformation $T_f: X \times S \rightarrow X \times S$ by
\begin{equation}\label{eq:skewprod}
T_f(x, i) := (Tx, f_i(x)).
\end{equation}
Define the \emph{deterministic walk on} $S$ \emph{in the environment} $(f_k)_{k \in S}$ to be
\begin{equation}\label{eq:detwalk}
U_n := U_{i,n}(x) := \pi_2(T_f^n(x, i))
\end{equation} 
where $\pi_2(x,y) := y$. 
\end{defn}

We will often refer to the set $S$ as either the \emph{state space} or the \emph{fibre}. We will also refer to the probability space $(X,m)$ as the \emph{base}, and the map $T:X \rightarrow X$ as the \emph{base map}.

The following properties of a deterministic walk are fundamental to the exposition of our main results.

\begin{defn}\label{def_RecTransofDW}
Given a deterministic walk on a countable state space $S$ in an environment $(f_k)_{k \in S}$, we say that a state $i \in S$ is \emph{recurrent} if 
\[
m(\{x \in X: U_{i,n}(x) = i, \mathrm{for~some}~n \geq 1\}) = 1. 
\]
We say that a state $i$ is \emph{transient} if it is not recurrent. 

We say that the deterministic walk is \emph{transitive on} $S$ if for all $i, j \in S$, 
\[
m(\{x \in X: U_{i,n}(x) = j, \mathrm{for~some}~n \geq 1\}) = 1. 
\]
\end{defn}

\begin{defn}
In the setting of the deterministic walk we consider the analogous setup to the RWRE in which the environment of transition functions on the state space must first be generated before the deterministic walk is run. In Sections \ref{Sec:Zero_One} and \ref{Sec:Transience}, our primary interest in such systems will be in the situation where the state space is $\Z$, and where environments $(\ldots, f_{-1}, f_0, f_1, \ldots)$ of transition functions are generated by an i.i.d.\ process or, more generally, an ergodic and stationary process. Because such processes can be described deterministically, we call such a system a \emph{deterministic walk in a deterministic environment} on $\Z$ (DWDE).
\end{defn}

\section{Asymptotic Properties of a Deterministic Walk in a Fixed Environment}\label{Sec:AsympPropsDWDE}

In this section we prove our first main result, in which we establish hypotheses under which a deterministic walk in a given fixed environment on a countable state space $S$ is recurrent with probability 0 or 1. 

\begin{defn}
Define the measure $\mu := m \times \mathrm{counting~measure}$.
\end{defn}
Clearly, $\mu$ is a $\sigma$-finite measure. 

The following two results are straightforward consequences of the standing hypotheses, and are of central importance. (The routine details of the proofs of the results can be found in \cite[Proposition 4.1]{Little2012a} and \cite[Proposition 4.3]{Little2012a}, respectively.)
\begin{pro}\label{pro:SkewMarkov}
Suppose that $T$ is Markov, with Markov partition $\beta$, and that transition functions are constant on elements of $\beta$. Then for any realisation of the environment the skew-product $T_f: X \times S \rightarrow X \times S$ is a Markov transformation of the measure space $(X \times S, \mu)$ with Markov partition $\beta \times S$.
\end{pro}

\begin{pro}\label{pro:T_fisMGF}
Suppose that $T$ is Markov, with Markov partition $\beta$, and that it has the Strong Distortion Property. Suppose also that transition functions are constant on elements of $\beta$. Then the skew-product $T_f: X \times S \rightarrow X \times S$ has the Strong Distortion Property with respect to the measure $\mu$.
\end{pro}

\begin{defn}
We say that a Markov transformation $T$ of probability space $(X, m)$, with Markov partition $\beta$, has the \emph{Big Image Property} if $\inf_{a \in \beta} m(Ta) > 0$.
\end{defn}
\textbf{Standing hypotheses for the deterministic walk}. Hereafter, we assume that the deterministic walk is driven by a Markov transformation $T$ of a probability space $(X, m)$, with Markov partition $\beta$, that has the Strong Distortion and Big Image properties. We also assume that transition functions are constant on elements of $\beta$.

\begin{defn}
Define $\tilde{\beta} := \beta \times S$ and, for $n \geq 1$, $\tilde{\beta}_n := \bigvee_{j=0}^{n-1}T_f^{-j}\tilde{\beta}$.
\end{defn}
It is an immediate consequence of the hypotheses that $T$ is a Markov map with Markov partition $\beta$ and that transition functions are constant on elements of $\beta$, that for all $n \geq 1$ and all $v \in \tilde{\beta}_n$ there exists $a \in \beta_n$, $i, k \in S$ such that 
\begin{equation}\label{eq:vn_A}
v  = a \times \{i\} \quad \& \quad T_f^n(v) = T^na \times \{k\}.
\end{equation}
It follows from (\ref{eq:vn_A}) that for all cylinders $v$
\begin{equation}\label{eqn:bddimage}
\mu(T_f^{|v|}v) \leq 1.
\end{equation}
It also follows from (\ref{eq:vn_A}) and from the Big Image property that
\begin{equation}\label{eqn:bigimage1}
\mu(T_f^{n}v) = m(T^na) \geq \ep > 0
\end{equation}
where $\ep: = \inf_{a \in \beta}m(a)$.

\begin{defn}\label{def:bj} We say that the deterministic walk has \emph{bounded jumps} if $f(X)$ is a finite subset of $S$, for all transition functions $f$.
\end{defn}

\begin{defn} For $A \subset \tilde{\beta}$, we define the projection of $A$,
\[
\pi_2(A) := \{i \in S: \exists a \in \beta: a \times \{i\} \in A\}.
\]
Given $K \subset S$, $i \in S$, and $x \in X$, we say that that the orbit $(U_{i,n}(x))_{n \geq 0}$ is \emph{transitive} on $K$ if it visits every element of $K$ infinitely often. Conversely, we say that the orbit $(U_{i,n}(x))_{n \geq 0}$ is \emph{transient} on $K$ if it visits every element of $K$ at most finitely often.
\end{defn}

\begin{thm}\label{thm:Transience_TransitivityofOrbits}
If a deterministic walk in a fixed environment has bounded jumps, then for any communication class $\mathcal{C} \subset \tilde{\beta}$, either (i) for all $a \in \mathcal{C}$ and $\mu$-a.e.\ $(x, i) \in a$ the orbit $(U_{i,n}(x))_{n \geq 0}$ is transitive on $\pi_2(\mathcal{C})$, or (ii) for all $a \in \mathcal{C}$ and $\mu$-a.e.\ $(x, i) \in a$ the orbit $(U_{i,n}(x))_{n \geq 0}$ is transient on $\pi_2(\mathcal{C})$.
\end{thm}

\begin{proof}
It follows from Lemma \ref{cor:Transience} and Proposition \ref{pro:T_fisMGF} that the communication class $\mathcal{C}$ is either recurrent or transient. 

Suppose that $\mathcal{C}$ is recurrent. Then by Lemma \ref{cor:Transitivity}, $\mathcal{C}$ is transitive, and it follows that for all $a \in \mathcal{C}$ and $\mu$-a.e.\ $(x, i) \in a$, $(U_{i,n}(x))_{n \geq 0}$ is transitive on $\pi_2(\mathcal{C})$.

Conversely, suppose that $\mathcal{C}$ is transient. We deal separately with the cases where $\#\beta < \infty$ and $\#\beta = \infty$.

Suppose that $\# \beta < \infty$. An immediate consequence of Lemma \ref{lem:Trans_iozero} and Proposition \ref{pro:T_fisMGF} is that for all $a \in \mathcal{C}$ and $\mu$-a.e.\ $(x, i) \in a$, the orbit of $(x, i)$ under $T_f$ visits each partition element in $\mathcal{C}$ at most finitely often. For each $i \in S$, $X \times \{i\}$ contains only finitely many partition elements, and the result follows.

Suppose, instead, that $\#\beta = \infty$. Fix $i \in S$ such that 
\[
\mathcal{C}_i := \{a \in \mathcal{C}: a \subset X \times \{i\} \} \neq \emptyset.
\]
Define
\[
D_i := \{x: \exists a \in \mathcal{C}_i, x \in a\}. 
\]
To complete the proof we show that
\begin{equation}\label{eqn:transiozero}
\mu(\{(x, i) \in D_i: U_{i,n}(x) = i ~\mathrm{i.o.}\}) = 0.
\end{equation}
It then follows from the non-singularity of $T_f$ that for all $a \in \mathcal{C}$ and $\mu$-a.e.\ $(x, i) \in a$, the orbit $(U_{i,n}(x))_{n \geq 0}$ is transient on $\pi_2(\mathcal{C})$, as required. 

Since the deterministic walk has bounded jumps it follows that there exists a finite set $J \subset S$ such that for all $x \in D_i$, $\pi_2(T_f(x)) \in J$. For each $l \in J$, let $\{a_{l, n}\}_{n \in \N}$ be an enumeration of all the partition elements of $\beta$ such that $a_{l, n} \times \{l\} \subset T_f(D_i)$. By (\ref{eqn:bigimage1}) there exists $\ep > 0$ such that for all $a \in \tilde{\beta}$, $\mu(T_f(a)) \geq \ep$. For each $l \in J$ define
\[
M_l := \inf\{k: m(\cup_{n=k}^{\infty} a_{l, n}) < \frac{\ep}{2}\}.
\]
Define 
\[
\Lambda := \{(x, l) \in X \times S: l \in J ~\&~ x \in a, \mathrm{for~some}~ a \in \{a_{l, n}\}_{n=1}^{M_l}\}.
\]
Given an $n$-cylinder $v \in \tilde{\beta}_n$, we let $v_j$ denote the $j$th component of $v$, for $0 \leq j \leq n-1$. For $n \geq 1, q \geq 0$ and $i \in S$ define
\begin{eqnarray*}
B_{q,n} := \{v \in \tilde{\beta}_n: \exists 0 = r_0 < r_1 < \ldots < r_q = n-1, v_{r_k} \in \mathcal{C}_i, ~\mathrm{for}~k = 0, \ldots, q,& &\\
~\&~ v_{t} \cap \Lambda = \emptyset, ~\mathrm{for}~t = 0, \ldots, n-1\}.
\end{eqnarray*}
Thus, each $v \in B_{q,n}$ is an $n$-cylinder contained in the set $D_i$, each of whose elements make at least $q$ return visits to $\mathcal{C}_i$ before visiting $\Lambda$, if ever. 

It follows from Proposition \ref{pro:GibbsPropimpliesLeakage} and Proposition \ref{pro:T_fisMGF} that there exists $C \geq 1$ such that for all $n \geq 1$ and $q \geq 0$, and for all $v := [c_0 \times \{i\}, \ldots, c_{n-1} \times \{i\}] \in B_{q, n}$, we have  
\begin{equation}\label{eqn:Lambdadecay1}
\frac{\mu(v \cap T_f^{-n}(\Lambda))}{\mu(v)}  \geq C^{-1} \frac{\mu(T_f^n(v) \cap \Lambda))}{\mu(T_f^n(v))}.
\end{equation}
By (\ref{eq:vn_A}), $T_f^n(v) = T_f(c_{n-1} \times \{i\}) = Tc_{n-1} \times \{k\}$, where $k := f_i(c_{n-1})$. From the definition of $\mu$ and $\Lambda$ we have 
\begin{equation}\label{eqn:Lambdadecay2}
\frac{\mu(T_f^n(v) \cap \Lambda)}{\mu(T_f^n(v))} = \frac{\mu(Tc_{n-1} \times \{k\} \cap (\cup_{j=1}^{M_k} a_{k, j} \times \{k\}))}{\mu(Tc_{n-1} \times \{k\})} = \frac{m(Tc_{n-1}  \cap \cup_{j=1}^{M_k} a_{k, j})}{m(Tc_{n-1})}.
\end{equation}
\begin{equation}\label{eqn:Lambdadecay3}
\frac{m(Tc_{n-1}  \cap \cup_{j=1}^{M_k} a_{k, j})}{m(Tc_{n-1})}  \geq m(Tc_{n-1} \cap \cup_{j=1}^{M_k}a_{k, j}) \geq \ep - \frac{\ep}{2} = \frac{\ep}{2}.
\end{equation}
It now follows from (\ref{eqn:Lambdadecay1}), (\ref{eqn:Lambdadecay2}) and (\ref{eqn:Lambdadecay3}) that
\begin{equation}\label{eqn:leakofT_f}
\frac{\mu(v \cap T_f^{-n}(\Lambda))}{\mu(v)} \geq C^{-1} \frac{\ep}{2} > 0.
\end{equation}
Having established (\ref{eqn:leakofT_f}), a similar argument to that given in the proof of Lemma \ref{lem:Trans_iozero} shows that the measure of walks that start in $D_i$ and that make $q$ returns to $\mathcal{C}_i$, without visiting the set $\Lambda$, decays exponentially in $q$ at a rate of at least $1-C^{-1}\frac{\ep}{2}$. Defining 
\[
A: = \{(x, i) \in D_i: U_{i,n}(x) = i~\mathrm{i.o.}~\&~ \forall n \geq 1,~T_f^n(x, i) \notin \Lambda\},
\]
it follows immediately that
\begin{equation}\label{eqn:mDzero}
\mu(A) = 0.
\end{equation} 
Defining 
\[
B := \{(x, i) \in D_i: U_{i,n}(x) = i~\mathrm{i.o.}~\&~ \exists N \forall n \geq N,~T_f^n(x, i) \notin \Lambda\},
\]
it follows that $B \subset D_i \cap (\cup_{n \geq 0} T_f^{-n}A)$. Since $T_f$ is non-singular it follows from (\ref{eqn:mDzero}) that
\begin{equation}\label{eqn:io0}
\mu(B) = 0.
\end{equation} 
Since the deterministic walk has bounded jumps $\# \Lambda < \infty$. It follows that since $\mathcal{C}$ is transient, for $\mu$-a.e.\ $(x, i) \in D_i$ there exists $N$ such that for all $n \geq N$, $T_f^n(x, i) \notin \Lambda$. Equation (\ref{eqn:transiozero}) now follows from (\ref{eqn:io0}). This completes the proof.
\end{proof}

We have the following consequence of Theorem \ref{thm:Transience_TransitivityofOrbits}.

\begin{thm}\label{cor:NoSemibdd}
Suppose that $S = \Z$, the deterministic walk in a given environment has bounded jumps, and that for every closed communication class $\mathcal{C} \subset \tilde{\beta}$, $\pi_2(\mathcal{C}) = \Z$. Then either
\begin{itemize}
\item[(i)] the deterministic walk is recurrent and $\limsup_{n \rightarrow \infty} U_{i,n}(x) = \infty$ and $\liminf_{n \rightarrow \infty} U_{i,n}(x) = -\infty$, for all $a \in \mathcal{C}$ and $\mu$-a.e.\ $(x, i) \in a$, or
\item[(ii)] $\lim_{n \rightarrow \infty} |U_{i,n}(x)| = \infty$, for all $a \in \mathcal{C}$ and $\mu$-a.e.\ $(x, i) \in a$.	
\end{itemize} 
\end{thm}
\begin{proof}
It is immediate from Theorem \ref{thm:Transience_TransitivityofOrbits} that for each closed communication class $\mathcal{C} \subset \tilde{\beta}$, $(U_{i,n}(x))_{n \geq 0}$ is recurrent on $\pi_2(\mathcal{C}) = \Z$ for all $a \in \mathcal{C}$ and $\mu$-a.e.\ $(x,i) \in a$, or $(U_{i,n}(x))_{n \geq 0}$ is transient on $\pi_2(\mathcal{C}) = \Z$ for all $a \in \mathcal{C}$ and $\mu$-a.e.\ $(x,i) \in a$. The recurrent case implies that (i) holds, whereas the transient case implies that (ii) holds.
\end{proof}

\begin{rmk} In the transient case, the limit $\pm \infty$ may depend on $x$. For example, in the case of a Markov chain on $\Z$ with jumps of $\pm 1$ such that $p_{i, i+1} = \frac{3}{4}$ and $p_{i, i-1} = \frac{1}{4}$ for all $i \geq 0$, and $p_{i, i+1} = \frac{1}{4}$ and $p_{i, i-1} = \frac{3}{4}$ for all $i < 0$, then the walk diverges both to the left and to the right with positive probability.
\end{rmk}

\section{Zero-One Laws for a DWDE on $\Z$}\label{Sec:Zero_One}

Hereafter, we confine our attention to the situation where the state space $S = \Z$. In addition to the standing hypotheses we will also assume henceforth that the Markov partition $\beta$ separates points (i.e. for all $x \neq y \in X$, there exists $n \geq 1$ and $a, b \in \beta_n$ such that $a \cap b = \emptyset$ and $x \in a$ and $y \in b$). We also assume that the process by which the environment is generated is \emph{ergodic} and \emph{stationary}. More precisely, we consider the following setup. 

As per the standing hypotheses, we suppose that the set of transition functions from which environments are constructed is some set
\begin{equation}\label{G}
G \subseteq \{g: X \rightarrow \Z: g~\mathrm{constant~on~elements~of~\beta}\}, 
\end{equation} 
Let $\eta: \Omega \rightarrow \Omega$ be an invertible measurable transformation of a probability space $(\Omega, P)$, and $\phi: \Omega \rightarrow G$ be a measurable function. We suppose that each point $\omega \in \Omega$ encodes an environment $(f_i)_{i \in \Z}$ such that
\begin{equation}
f_i := \phi \circ \eta^i(\omega), \quad \mathrm{for~all}~i \in \Z.
\end{equation} 
We call the process $(\phi \circ \eta^i)_{i \in \Z}$ \emph{a deterministic environment on} $\Z$. 

We shall be interested in the cases where where the deterministic environment is either i.i.d.\ or, more generally, ergodic and stationary (i.e.\ $\eta$ is ergodic and measure preserving).

Associated to every $\omega \in \Omega$ is a function 
\begin{eqnarray*}
f(\omega): X \times \Z & \rightarrow & \Z\\
             (x, i)   & \mapsto & f_i(x)
\end{eqnarray*}
where $f_i := \phi(\eta^i(\omega))$. For $\omega \in \Omega$, we thus define the skew-product $T_{f(\omega)}: X \times \Z \rightarrow X \times \Z$ such that
\begin{equation}
T_{f(\omega)}(x, i) = (Tx, i + f_i(x)).
\end{equation}
Typically, we will suppress reference to $\omega$, the dependence being tacit, and simply write $T_f$ instead of $T_{f(\omega)}$.

\begin{defn}\label{def:DWDE}
We define a \emph{deterministic walk in a deterministic environment on} $\Z$ (DWDE) to be the process $U_n := \pi_2 \circ T_f^n$, for $n \geq 0$. Specifically, for $i \in \Z$, $x \in X$, and $\omega \in \Omega$, we define 
\begin{equation}\label{eq:DWDE1}
U_{i,n}(x, \omega) := \pi_2(T_{f(\omega)}^n(x, i)).
\end{equation}
It follows that $U_{i,0}(x, \omega) := i$ and for all $n \geq 0$
\begin{equation}\label{eq:DWDE2}
U_{i,n+1}(x, \omega) := U_{i,n}(x, \omega) + f(\omega)(T^n(x), U_{i,n}(x, \omega)).
\end{equation}
Hence, $U_{i,n+1}(x, \omega) = \sum_{r=0}^n f(\omega)(T^r(x), U_{i,r}(x, \omega))$.
\end{defn}
Before proceeding to the main results of this section we establish some technical results regarding the measurability of certain sets of environments - specifically, Lemma \ref{pro:MeasurabilityofDWDE} and Corollary \ref{cor:MeasR} below. (The reader who is not concerned with such technicalities may wish to pass over these results.) It follows from (\ref{G}) that if $\# \beta < \infty$ then $G$ is necessarily countable, in which case $G$ is equipped with the discrete topology and we assume that $\phi: \Omega \rightarrow G$ is measurable with respect to the sigma-algebra generated by that topology. If instead $\#\beta = \infty$, and $\{a_n\}_{n \geq 1}$ is an enumeration of the Markov partition $\beta$, define the function $S: G \times G \rightarrow \N$ such that $S(f,g) := \max\{n: f(a_n) = g(a_n)\}$. Defining the metric $d: G \times G \rightarrow [0,1]$ such that $d(f, g) := (\textstyle\frac{1}{2})^{S(f,g)}$, we assume that $\phi: \Omega \rightarrow G$ is measurable with respect to the sigma-algebra generated by the topology given by $d$.

\begin{lem}\label{pro:MeasurabilityofDWDE}
Given $n \geq 0$, $k_0, \ldots, k_{n-1} \in \Z$ and $r \in [0,1]$, the set $\{\omega: m(\{x: U_{0,s}(x, \omega) = \sum_{i=0}^{s-1}k_i, ~\mathrm{for}~s = 1, \ldots, n\}) > r\}$ is measurable.
\end{lem}
\begin{proof}
Fix $n \geq 0$, $k_0, \ldots, k_{n-1} \in \Z$, and $r \in [0,1]$. Fix $a \in \beta_n$ and define
\[
\Pi_a := \{\omega: U_{0,s}(x, \omega) = \sum_{i=0}^{s-1}k_i,~\forall x \in a ~\&~ s = 0, \ldots, n\}.
\]
Firstly, we show that $\Pi_a$ is measurable. In both the case where $\#\beta < \infty$ and $\#\beta = \infty$ it is apparent that for all $k \in \Z$ and all $b \in \beta$, $\{g: g(b) = k\}$ is an open set, and hence $\phi^{-1}\{g: g(b) = k\}$ is measurable. Since $\eta: \Omega \rightarrow \Omega$ is measurable it follows that for all $i, k \in \Z$ and all $b \in \beta$ the set $\{\omega: (\phi(\eta^i(\omega))(b) = k\}$ is measurable. If $a: = [a_0, \ldots, a_{n-1}]$, letting $R_s := \sum_{i=0}^{s-1}k_i$ since 
\[
\Pi_a = \{\omega: \phi(\eta^{R_s}\omega)(a_s) = k_s,~ \mathrm{for}~ s = 0, \ldots, n-1\}
\]
it follows that $\Pi_a$ is measurable.  

A given $\omega \in \Omega$ satisfies $m(\{x: U_{0,s}(x, \omega) = R_s, ~\mathrm{for}~s = 0, \ldots, n\}) > r$
if and only if there exists a finite collection $K$ of $n$-cylinders $a$ such that $m(\cup_{a \in K} a) > r$ and $\omega \in \cap_{a \in K}\Pi_a$. Letting $\Delta$ denote the set of all finite collections $K$ of $n$-cylinders such that $m(\cup_{a \in K}a) > r$ it follows that
\begin{equation}\label{eq:Mes1}
\{\omega: m(\{x: U_{0,s}(x, \omega) = R_s, ~\mathrm{for}~s = 0, \ldots, n\}) > r\} = \bigcup_{K \in \Delta}\bigcap_{a \in K}\Pi_a.
\end{equation}
Since $\Delta$ is countable, and it follows that $\cup_{K \in \Delta}\cap_{a \in K} \Pi_a$ is measurable, and so (\ref{eq:Mes1}) establishes the result.
\end{proof}

As a consequence of Lemma \ref{pro:MeasurabilityofDWDE} we have the following corollary.
\begin{cor}\label{cor:MeasR}
Given $n \geq 1$, $l, k \in \Z$, and $r \in [0, 1]$, the set $\{\omega: m(\{x: U_{l,j}(x, \omega) \leq k, \mathrm{for~some}~1 \leq j \leq n\}) > r\}$ is measurable.
\end{cor}
Define $\tilde{\beta} := \beta \times \Z$.

\begin{defn}
We say that the DWDE is \emph{irreducible} if for $P$-a.e.\ $\omega \in \Omega$, the Markov partition $\tilde{\beta}$ is irreducible under the skew-product $T_{f(\omega)}$.
\end{defn}

\begin{defn}
We say that the DWDE is \emph{recurrent} if for $P$-a.e.\ $\omega \in \Omega$, the deterministic walk in the environment $(\ldots, \phi(\eta^{-1}\omega), \phi(\omega), \phi(\eta\omega), \ldots)$ is recurrent. Similarly, we say that the DWDE is \emph{transient} if for $P$-a.e.\ $\omega \in \Omega$, the deterministic walk in the environment $(\ldots, \phi(\eta^{-1}\omega), \phi(\omega), \phi(\eta\omega), \ldots)$ is transient.
\end{defn}
Recall that the DWDE has \emph{bounded jumps} if for all $g \in G$, the set $g(X)$ is finite.

We may now state and prove the first main result of this section.
\begin{thm}\label{thm:4Case0-1Law}
If the DWDE is irreducible, has bounded jumps, and has an ergodic and stationary environment, then it is either recurrent or transient. Moreover, in the transient case, exactly one of the following holds.
\begin{enumerate}
\item[(i)] $\lim_{n \rightarrow \infty} U_{i,n}(x, \omega) = + \infty$ for all $i \in \Z$, for $P$-a.e.\ $\omega \in \Omega$ and for $m$-a.e.\ $x \in X$.
\item[(ii)] $\lim_{n \rightarrow \infty} U_{i,n}(x, \omega) = - \infty$ for all $i \in \Z$, for $P$-a.e.\ $\omega \in \Omega$ and for $m$-a.e.\ $x \in X$.
\item[(iii)] For $P$-a.e.\ $\omega \in \Omega$, and for all $i \in \Z$,
\[
m(\{x: \lim_{n \rightarrow \infty} U_{i,n}(x, \omega) = \infty\}) + m(\{x: \lim_{n \rightarrow \infty} U_{i,n}(x, \omega) = -\infty\}) = 1,
\]
and $0 < m(\{x: \lim_{n \rightarrow \infty} U_{i,n}(x, \omega) = \infty\}) < 1$.
\end{enumerate}
\end{thm}

\begin{proof}
Let $R^+ := \{\omega: m(\{x: \lim_{n \rightarrow \infty} U_{0,n}(x, \omega) = + \infty\}) = 0\}$. Thus, $R^+$ denotes the set of environments for which the measure of the set of walks that diverge to $+\infty$ when started in state 0, is zero. It follows from Corollary \ref{cor:MeasR} that $R^+$ is $P$-measurable. We show that if $\omega \in R^+$ then for all $k \in \Z$, 
\begin{equation}\label{eq:limk0}
m(\{x: \lim_{n \rightarrow \infty} U_{k,n}(x, \omega) = \infty\}) = 0.
\end{equation}
Suppose to the contrary that $\omega \in R^+$ and that there exist $k \neq 0$ such that 
\[
m(\{x: \lim_{n \rightarrow \infty} U_{k,n}(x, \omega) = \infty\}) > 0.
\]
Then there exists $a \in \beta$ and a set $A \subset a$, such that $m(A) > 0$ and for all $x \in A$ 
\[
\lim_{n \rightarrow \infty} U_{k,n}(x, \omega) = \infty.
\] 
By the irreducibility of $\tilde{\beta}$, for each $b \in \beta$ there exists an admissible cylinder $[b \times \{0\}, \ldots, a \times \{k\}]$. It follows from the non-singularity of $T_f$, that $\mu([b \times \{0\}, \ldots, A \times \{k\}]) > 0$, but this contradicts the fact that $\omega \in R^+$, and so we have established (\ref{eq:limk0}).

Using the fact that for all $k \in \Z$ and for all $x \in X$
\begin{equation}\label{eq:itin0k}
U_{k,n}(x, \omega) = U_{0,n}(x, \eta^k (\omega)) + k
\end{equation}
it is now immediate from (\ref{eq:limk0}) and (\ref{eq:itin0k}) that if $\omega \in R^+$ then $\eta^k (\omega) \in R^+$ for all $k \in \Z$. In particular, $\omega \in R^+$ if and only if $\eta(\omega) \in R^+$, from which it follows that $\eta^{-1}(R^+) = R^+$. Since the $\eta$-invariant measure $P$ is ergodic, it follows from Birkhoff's Ergodic Theorem that $P(R^+) = 0~\mathrm{or}~1$.
 
An identical argument shows that the set 
\[
R^- : = \{\omega: m(\{x: \lim_{n \rightarrow \infty} U_{0,n}(x, \omega) = - \infty\}) = 0\}
\]
has $P$-measure of 0 or 1. 

If $P(R^+) = P(R^-) = 1$, then since the DWDE has bounded jumps it follows from Theorem \ref{cor:NoSemibdd}(i) that the DWDE is recurrent. Alternatively, in the transient cases we have the following
\begin{eqnarray*}
P(R^+) = 0,~P(R^-) = 1 & \Rightarrow & ~\mathrm{case~(i)}\\
P(R^+) = 1,~P(R^-) = 0 & \Rightarrow & ~\mathrm{case~(ii)}\\
P(R^+) = 0,~P(R^-) = 0 & \Rightarrow & ~\mathrm{case~(iii)}.
\end{eqnarray*}
This concludes the proof.
\end{proof}

\begin{defn}
Given a collection of cylinder sets $v_1, v_2, \ldots, v_k$ we define
\[
v_1\cdot v_2 \cdot \ldots \cdot v_k := v_1 \cap T^{-|v_1|}(v_2 \cap T^{-|v_2|}(\ldots (T^{-|v_{k-1}|}(v_k))\ldots )).
\]
\end{defn}

\begin{defn}\label{def:Linkage}
We say that a DWDE has the \emph{Linkage Property} if it is irreducible and there exists $r > 0$ such that for almost every environment and all $a, b \in \tilde{\beta}$ such that $|\pi_2(a)-\pi_2(b)| = 1$, either $b \subset T_f(a)$, or there exists a cylinder $c$ such that
\begin{itemize}
\item[(a)] $\mu(c) \geq r$, and
\item[(b)] the cylinder $a \cdot c \cdot b$ is admissible.
\end{itemize}
\end{defn}
The Linkage Property holds for a large class of DWDEs. 
\begin{examp}\label{ex:Linkage}
Suppose that a DWDE on $\Z$ satisfies (i) $ 2 \leq \# \beta < \infty$ and (ii) for all $g \in G$, $g(X) = \{+1, -1\}$, and that the base transformation $T$ has full-branches (i.e.\ $Ta = X$ for all $a \in \beta$). Then the DWDE has the Linkage Property. 

To see this we fix an environment $(f_i)_{i \in \Z}$, and fix $a = a' \times \{i\} \in \tilde{\beta}$, and suppose that $f_i(a') = +1$. (The case where $f_i(a') = -1$ is identical.) If $b = b' \times \{i+1\}$ for some $b' \in \beta$ then, since $T$ is full-branch, it is immediate that $T_f(a) = X \times \{i + 1\} \supset b$. Instead, suppose that $b = b' \times  \{i-1\}$ for some $b' \in \beta$. By (ii) there exists $c_1 = c'_1 \times \{i+1\}, c_2 = c'_2 \times \{i\} \in \tilde{\beta}$ such that $f_{i+1}(c'_1) = f_i(c'_2) = -1$. Since $T$ is full-branch it follows that $a \cdot c_1 \cdot c_2 \cdot b$ is admissible, and since $\# \beta < \infty$ it follows that the DWDE has the Linkage Property, as claimed.
\end{examp}

\begin{pro}\label{pro:Linkage}
If the DWDE has the Linkage Property, then there exist positive numbers $\{r_k\}_{k \geq 0}$ such that for every environment, and for all $a, b \in \tilde{\beta}$ such that $|\pi_2(a) - \pi_2(b)| \leq k$, there exists a cylinder $c$ such that $\mu(c) \geq r_k$ and $a \cdot c \cdot b$ is admissible.
\end{pro}
\begin{proof}
Fix an arbitrary environment. Given $a \in \tilde{\beta}$ let $\pi_2(a)$ denote  the $\Z$-component of $a$, and $a'$ denote the $\beta$-component of $a$. Since the case $|\pi_2(a) - \pi_2(b)| = 1$ is immediate from the Linkage Property, we deal in turn with the cases where $\pi_2(b) - \pi_2(a) = 0$ and $\pi_2(b) - \pi_2(a) = k \geq 2$ - the case where $\pi_2(a) - \pi_2(b) = k \geq 2$ is proven similarly. Fix $d' \in \beta$ and for $k \in \Z$ define $d_k := d' \times \{k\}$. Define $\kappa := m(d')$.

Let $\pi(b) - \pi(a) = 0$ and suppose that $f_{\pi(a)}(a') = +1$. The case where $f_{\pi(a)}(a') = -1$ is identical. By the Linkage Property there exist cylinders $v_1, v_2$ such that $\mu(v_1) > r$ and 	$\mu(v_2) > r$, and $a \cdot v_1 \cdot d_{\pi(a)+1} \cdot v_2 \cdot b$ is admissible. By Proposition \ref{pro:T_fisMGF}, Proposition \ref{pro:GibbsPropimpliesLeakage}, and equation (\ref{eqn:bddimage})
\begin{eqnarray*}
\mu(v_1 \cdot d_{\pi(a)+1} \cdot v_2) & \geq & C^{-1} \mu(v_1) \mu(d_{\pi(a)+1} \cdot v_2)\\
                                    & \geq & C^{-2}\mu(v_1)\mu(d_{\pi(a)+1})\mu(v_2)\\
                                    & \geq & C^{-2}r^2\kappa > 0.
\end{eqnarray*}
This establishes the result for the case $\pi_2(b) - \pi_2(a) = 0$.

Assume that $\pi_2(b) - \pi_2(a) = k \geq 2$. By a similar argument to the case where $\pi_2(b) - \pi_2(a) = 0$ above, there exist cylinders $v_1, \ldots, v_{k}$, such that $\mu(v_i) > r$ for $i = 1, \ldots, k$, and $a \cdot v_1 \cdot d_{\pi(a)+1} \cdot v_2 \cdot \ldots \cdot v_{k-1} \cdot d_{\pi(a)+k-1} \cdot v_k \cdot b$ is admissible. As before, by Proposition \ref{pro:T_fisMGF}, Proposition \ref{pro:GibbsPropimpliesLeakage} and (\ref{eqn:bddimage}) we have
\begin{eqnarray*}
\mu( v_1 \cdot d_{\pi(a)+1} \cdot \ldots \cdot v_{k-1} \cdot d_{\pi(a)+k-1} \cdot v_k) & \geq & (C^{-1})^{2k-2} \mu(v_1) \ldots \mu(v_{k}) \kappa^{k-1}\\
                                  									       & \geq & (C^{-1})^{2k-2} r^{k} \kappa^{k-1} > 0.
\end{eqnarray*}
Taking $r_k := (C^{-1})^{2k-2} r^{k} \kappa^{k-1}$ establishes the result.
\end{proof}

\begin{defn}
We say that the DWDE has \emph{uniformly bounded jumps} if there exists a finite set $J \subset \Z$ such that for all $g \in G$, $g(X) \subset J$.
\end{defn}
We are now in a position to state and prove our main result.

\begin{thm}\label{thm:nosplitcase}
Suppose that the DWDE satisfies the Linkage Property with uniformly bounded jumps and an ergodic and stationary environment. If the DWDE is transient then either case (i) or case (ii) of Theorem \ref{thm:4Case0-1Law} holds. 
\end{thm}
Before outlining our strategy for proving Theorem \ref{thm:nosplitcase} we introduce the following prerequisite definitions.

\begin{defn}
Let $M := \max\{|g(x)|: x \in X, g \in G\}$, and for $j \in \Z$ define
\begin{equation}\label{eqn:Lambda_j}
\Lambda_j := \{(x, n) \in X \times \Z: jM \leq n \leq (j+1)M - 1\}.
\end{equation}
\end{defn}
Clearly, for all $i \in \Z$, $\mu(\Lambda_i) = M$.

\begin{defn}
Define 
\[
D^+ := \{\omega \in \Omega: \mu(\{(x, i) \in \Lambda_0: \lim_{n \rightarrow \infty} U_{i,n}(x, \omega) = + \infty\}) > 0\}.
\]
\end{defn}
The set $D^+$ denotes the set of environments for which the deterministic walk diverges to $+\infty$ with positive probability. Thus, $D^+ = (R^+)^c$ and it follows that $D^+$ is also $P$-measurable. 

Analogously, define 
\[
D^- := \{\omega \in \Omega: \mu(\{(x, i) \in \Lambda_0: \lim_{n \rightarrow \infty} U_{i,n}(x, \omega) = - \infty\}) > 0\}
\]
and by identical reasoning, $D^- = (R^-)^c$ and so $D^-$ is $P$-measurable. Moreover, since the hypotheses of Theorem \ref{thm:nosplitcase} imply those of Theorem \ref{thm:4Case0-1Law} it follows that $P(D^+) = 0~\mathrm{or}~1$ and $P(D^-) = 0~\mathrm{or}~1$.

\begin{defn}
For all $i, j, k \in \Z$, $n \geq 0$, and $\omega \in \Omega$ define taboo-hitting-time sets
\[
_jA_{i,k}^n(\omega) := \{x \in \Lambda_i: T_f^n(x) \in \Lambda_k, ~\& ~T_f^r(x) \notin \Lambda_k \cup \Lambda_j ~\mathrm{for}~r = 1, \ldots, n-1\},
\]
Also define
\[
_jA_{i,k}(\omega) := \cup_{n \geq 1} (_jA_{i,k}^n(\omega)).
\]
\end{defn}
Under the standing hypotheses, for all $\omega \in \Omega$ the set $_jA^n_{i, k}(\omega)$ is a union of cylinders in $\cup_{r \geq 1}\tilde{\beta}_r$, and is therefore $\mu$-measurable. By extension, for all $\omega \in \Omega$ the set $_jA_{i, k}(\omega)$ is also $\mu$-measurable.

\begin{defn} 
Define 
\[
E^+ := \{\omega \in \Omega: \exists p > 0,~ \mathrm{s.t.~for~infinitely~many}~k \geq 1, ~\mu(_{-k}A_{-k,0}(\omega)) \geq p\}.
\]
\end{defn}
From this definition we have that
\begin{equation}\label{eq:E+}
E^+ = \cup_{d \geq 1} \cap_{k \geq 1} \cup_{N = k}^{\infty}\{\omega: \mu(_{-k}A_{-k, 0}(\omega)) \geq d^{-1}\},
\end{equation}
and it follows from Corollary \ref{cor:MeasR} that $E^+$ is $P$-measurable.

\paragraph{Strategy for proving Theorem \ref{thm:nosplitcase}}
The proof of Theorem \ref{thm:nosplitcase} contains two main parts. Firstly, we prove (in Lemma \ref{lem:NoDiv-inf}) that $E^+ \cap D^- = \emptyset$. Secondly, we show (in Lemma \ref{lem:PD+PE+}) that if $P(D^+) = 1$ then $P(E^+) > 0$. The result then follows from Theorem \ref{thm:4Case0-1Law}. Suppose to the contrary that case (iii) of Theorem \ref{thm:4Case0-1Law} holds. Then $P(D^+) = 1$ and $P(D^-) = 1$, and it follows from Lemma \ref{lem:PD+PE+} that $P(E^+) > 0$. But by Lemma \ref{lem:NoDiv-inf} it follows that $P(E^+) = 0$, which is the desired contradiction. Hence, case (iii) cannot hold, establishing the result.

It remains to prove Lemma \ref{lem:NoDiv-inf} and \ref{lem:PD+PE+}. Firstly, we establish the following useful result.

\begin{lem}\label{pro:DecompD+}
Suppose that the DWDE is irreducible and that it has uniformly bounded jumps. Then 
\begin{itemize}
\item[(a)] $\omega \in D^-$ if and only if $\mu(\cap_{k \geq 1} (_{0}A_{0,-k}(\omega))) > 0$, and 
\item[(b)] $\omega \in D^+$ if and only if $\mu(\cap_{k \geq 1} (_{0}A_{0,k}(\omega))) > 0$.  
\end{itemize}
\end{lem}
\begin{proof}
We prove (a) as the proof of (b) is similar. The ($\Leftarrow$) part follows directly from Theorem \ref{cor:NoSemibdd}. For the $(\Rightarrow)$ part, assume that $\mu(\cap_{k \geq 1} (_{0}A_{0,-k}(\omega))) = 0$. Since $T_f$ is non-singular with respect to $\mu$, it follows that 
\[
\mu(\bigcup_{j \geq 0}T_f^{-j}(\cap_{k \geq 1} (_{0}A_{0,-k}(\omega)))) = 0. 
\]
But 
\[
\{(x, i) \in \Lambda_0: \lim_{n \rightarrow \infty} U_{i,n}(x, \omega) = - \infty\} \subset \bigcup_{j \geq 0}T_f^{-j}(\cap_{k \geq 1} (_{0}A_{0,-k}(\omega))).
\]
This completes the proof.
\end{proof}

\begin{lem}\label{lem:NoDiv-inf} 
Suppose that the DWDE satisfies the Linkage Property and that it has uniformly bounded jumps. Then $E^+ \cap D^- = \emptyset$.
\end{lem}
\begin{proof}
We show that for each $\omega \in E^+$ there exists a strictly increasing sequence $(n_j)_{j \geq 1}$ of natural numbers, and $q \in (0, 1)$, such that 
\begin{equation}\label{eqn:expdecay}
\mu(_{0}A_{0,-n_j}(\omega)) \leq Mq^j.
\end{equation} 
The result then follows from Lemma \ref{pro:DecompD+}(a). 

For $d \geq 1$, define 
\[
E_d^+ = \{\omega:~ \mu(_{-n}A_{-n,0}(\omega)) \geq d^{-1} ~\mathrm{for~infinitely~many}~n \geq 1\}.
\]
By (\ref{eq:E+}) $E^+ = \cup_{d \geq 1} E_d^+$. For the remainder of the proof we fix $d \geq 1$ and fix $\omega \in E^+_d$, and write $_jA_{i,k}$ instead of $_jA_{i,k}(\omega)$. 

Since $\omega \in E_d^+$ there exists an infinite, strictly increasing, sequence of natural numbers $(n_j)_{j \geq 1}$ such that for each $j \geq 1$
\begin{equation}\label{eqn:keyIneq}
\mu(_{-n_j} A_{-n_j, 0}) \geq d^{-1}.
\end{equation}
Since the Markov partition $\beta$ separates points, it follows that for all $\ep > 0$ there exists $n \geq 1$ such that for all $a \in \beta_n$, $m(a) < \ep$. To see this, suppose to the contrary that there exists $\ep > 0$ such that for all $n \geq 1$ there exists $a \in \beta_n$ such that $m(a) > \ep$. It follows that there exists an infinite sequence of partition elements $\{a_n: a_n \in \beta, n \geq 1\}$ such that $m(\bigcap_{n \geq 1} a_n) > \ep$, and so there exists a positive measure set of points that are not separated by the partition $\beta$, contradicting the assumption that $\beta$ separates points.

By extension, the partition $\tilde{\beta}$ separates points in $X \times \Z$, and from this it is immediate that for all $\ep > 0$ there exists $n \geq 1$ such that $\mu(c) < \ep$ for all $c \in \tilde{\beta}_n$. It follows from Proposition \ref{pro:Linkage} that there exists $r > 0$ such that for all $a, b \in \tilde{\beta}$ for which $|\pi_2(a) - \pi_2(b)| \leq M$, there exists a cylinder set $c$ such that $\mu(c) > r$ and $a \cdot c \cdot b$ is admissible. Let $k_0$ be such that for all $n \geq k_0$ and all $a \in \tilde{\beta}_n$
\begin{equation}\label{eqn:r_L_upper}
\mu(a) < r.
\end{equation}
Since $(n_j)_{j \geq 1}$ is a strictly increasing sequence we may suppose, after passing to a subsequence if necessary, that for all $j \geq 1$
\begin{equation}\label{eqn:N_0}
n_{j+1} - n_j - 1 > k_0.
\end{equation}
It follows from (\ref{eqn:N_0}) that for all cylinders $c$ such that 
\[
c \subseteq ~ _{-n_j}A_{-n_j, -n_{j+1}}~\mathrm{or}~ c \subseteq ~ _{-n_j}A_{-n_j, -n_{j-1}}, 
\]
we have
\begin{equation}\label{eqn:Sup_r_L}
\mu(c) < r. 
\end{equation}
\paragraph{Claim} There exists $\delta > 0$ such that for all $j \geq 1$, and for each cylinder $b$ satisfying
\begin{equation}\label{eqn:b}
b \subset ~ _0A_{0, -n_j} ~ \& ~ T_f^{|b|}b \subset \Lambda_{-n_j},
\end{equation}
we have
\begin{equation}\label{ineqn:decayb}
\mu(b \cap T_f^{-|b|}(_{-n_{j+1}}A_{-n_j, 0})) \geq  \delta \mu(b).
\end{equation}

Assuming the above claim we complete the proof as follows. Taking unions over all cylinders $b$ satisfying (\ref{eqn:b}), it follows from (\ref{ineqn:decayb}) that
\begin{equation}\label{eqn1}
\mu(_{0}A_{0, -n_j} \backslash _{0}A_{0, -n_{j+1}}) \geq \delta \mu(_{0}A_{0, -n_j}).
\end{equation}
Since $_{0}A_{0, -n_{j+1}} \subset ~_{0}A_{0, -n_j}$, from (\ref{eqn1}) we have
\begin{equation}\label{eqn2}
\mu(_{0}A_{0, -n_{j+1}}) \leq (1 - \delta)\mu(_{0}A_{0, -n_j}).
\end{equation}
Taking $q = 1 - \delta \in (0, 1)$, establishes (\ref{eqn:expdecay}).

\paragraph{Proof of Claim}
Fix $j \geq 1$. Let $\{a_n\}_{n \geq 1}$ be an enumeration of all partition elements $a \in \tilde{\beta}$ such that $a \subset \Lambda_{-n_j}$. We first show that there exists $\gamma > 0$ such that for all $a \in \{a_n\}_{n \geq 1}$
\begin{equation}\label{eqn:gamma}
\mu(a \cap ~ _{-n_{j+1}}A_{-n_j, 0}) \geq \gamma \mu(a).
\end{equation}
Fix $a \in \{a_n\}_{n \geq 1}$. By the Linkage Property and Proposition \ref{pro:Linkage}, it follows that for each $t \in \N$, there exists a cylinder $D_t := [d_0, \ldots, d_{k-1}]$, such that $d_i \neq a ~\mathrm{or}~a_t$ for $i = 0, \ldots, k-1$, $a \cdot D_t \cdot a_t$ is admissible, and 
\begin{equation}\label{eqn:LinkLower}
\mu(D_t) > r.
\end{equation}
For $t \in \N$ define 
\begin{equation}\label{eqn:DefBt}
B_t : = a_t \cap ~ _{-n_j}A_{-n_j, 0},
\end{equation} 
and 
\begin{equation}\label{eqn:DefPt}
p_t := \frac{\mu(B_t)}{\mu(_{-n_j}A_{-n_j, 0})}.
\end{equation}
Since the partition element $a \subset \Lambda_{-n_j}$ and $T_f(a) \supset D_t$, it follows that the cylinder 
\[
D_t \subset \Lambda_{-n_j-1} \cup \Lambda_{-n_j} \cup \Lambda_{-n_j+1}.
\]
Furthermore, from (\ref{eqn:Sup_r_L}) and (\ref{eqn:LinkLower}) it follows that for all $i = 0, \ldots, k-1$, 
\[
d_i \cap (\Lambda_{-n_{j-1}} \cup \Lambda_{-n_{j+1}}) = \emptyset
\]
and hence for all $x \in a \cdot D_t$, 
\begin{equation}\label{eqn:NoVisitto0}
T_f^i(x) \notin \Lambda_{-n_{j-1}} \cup \Lambda_{-n_{j+1}} ~\mathrm{for}~ i = 0, \ldots, k.
\end{equation}
Thus, for all $t \geq 1$
\begin{equation}\label{eqn:Returnto-nj}
a \cdot D_t \cdot a_t \subset ~ _{-n_{j+1}}A_{-n_j, -n_j} \bigcap ~_{-n_{j-1}}A_{-n_j, -n_j} \subset~ _{-n_{j+1}}A_{-n_j, -n_j} \bigcap ~_{0}A_{-n_j, -n_j}.
\end{equation} 
From (\ref{eqn:DefBt}) and (\ref{eqn:Returnto-nj}) we have
\begin{equation}\label{eqn:Escapebackto0}
a \cdot D_t \cdot B_t \subset ~ _{-n_{j+1}}A_{-n_j, 0}.
\end{equation} 
Moreover, it follows from (\ref{eqn:NoVisitto0}) and (\ref{eqn:Escapebackto0}) that $a \cdot D_t \cdot B_t$ consists of points in $\Lambda_{-n_j}$ that start in $a$, and whose last visit to $\Lambda_{-n_j}$ before eventually visiting $\Lambda_0$ occurs in the partition element $a_t$. It follows that for all $t \neq s$
\begin{equation}\label{Disjoint}
a \cdot D_t \cdot B_t \cap a \cdot D_s \cdot B_s = \emptyset.
\end{equation}
It follows from Proposition \ref{pro:T_fisMGF}, Proposition \ref{pro:GibbsPropimpliesLeakage} and equation (\ref{eqn:bddimage}) that for all $t \geq 1$
\begin{equation}\label{ineq:Long1}
\frac{\mu(a \cdot D_t \cdot B_t)}{\mu(a)} \geq C^{-1} \frac{\mu(T_f(a) \cap D_t \cdot B_t)}{\mu(T_f(a))} \geq C^{-1}\mu(T_f(a) \cap D_t \cdot B_t).
\end{equation}
Since $T_f(a) \supset D_t$ we have
\begin{equation}\label{eq:n}
\mu(T_f(a) \cap D_t \cdot B_t) = \mu(D_t \cdot B_t).
\end{equation}
From $T^{|D_t|}D_t \supset a_t \supset B_t$ it follows from Proposition \ref{pro:T_fisMGF}, Proposition \ref{pro:GibbsPropimpliesLeakage} and equation (\ref{eqn:bddimage})
\begin{equation}\label{ineq:Long2}
\frac{\mu(D_t \cdot B_t)}{\mu(D_t)} \geq C^{-1}\frac{\mu(B_t)}{\mu(T_f^{|D_t|}D_t)} \geq C^{-1}\mu(B_t).
\end{equation}
It now follows from (\ref{ineq:Long1}), (\ref{eq:n}) and (\ref{ineq:Long2}) that
\begin{equation}\label{eq:z}
\frac{\mu(a \cdot D_t \cdot B_t)}{\mu(a)} \geq C^{-2} \mu(D_t) \mu(B_t)
\end{equation}
From (\ref{eqn:LinkLower}), (\ref{eqn:DefPt}) and (\ref{eq:z}) we obtain
\begin{equation}\label{ineq:Key1}
\frac{\mu(a \cdot D_t \cdot B_t)}{\mu(a)} \geq C^{-2} r p_t \mu(_{-n_j}A_{-n_j, 0}) \geq C^{-2}r p_t d^{-1} > 0.
\end{equation}
It follows from (\ref{Disjoint}) that $\mu(\bigcup_{t \in \N}(a \cdot D_t \cdot B_t)) = \sum_{t \in \N} \mu(a \cdot D_t \cdot B_t)$.
Taking unions over $t \in \N$, it follows from (\ref{ineq:Key1}) that
\begin{equation}\label{ineq:KeyIneq2}
\mu(\bigcup_{t \in \N}(a \cdot D_t \cdot B_t)) \geq \mu(a) \sum_{t \in \N}C^{-2} r p_t d^{-1} = (C^{-2}rd^{-1})\mu(a) > 0.
\end{equation}
Taking $\gamma := C^{-2}rd^{-1} > 0$ it follows from (\ref{eqn:Escapebackto0}) that
\begin{equation}
\mu(a \cap _{-n_{j+1}}A_{-n_j, 0}) \geq \mu(\bigcup_{t \in \N}(a \cdot D_t \cdot B_t)) \geq \gamma \mu(a).
\end{equation}
This establishes (\ref{eqn:gamma}).

The set $_0A_{0,-n_j}$ can be expressed as the union of cylinders satisfying (\ref{eqn:b}). We fix such a cylinder $b$. Since $T_f^{|b|}b$ is the union of partition elements in $\tilde{\beta}$, it is immediate from (\ref{eqn:gamma}) that 
\begin{equation}\label{eqn:Tb}
\mu(T_f^{|b|}b \cap _{-n_{j+1}}A_{-n_j, 0}) \geq \gamma \mu(T_f^{|b|}b).
\end{equation}
From Proposition \ref{pro:T_fisMGF}, Proposition \ref{pro:GibbsPropimpliesLeakage} and (\ref{eqn:Tb}) we have that
\begin{equation}\label{ineqn:Leakfromb}
\frac{\mu(b \cap T_f^{-|b|} ~ (_{-n_{j+1}}A_{-n_j, 0}))}{\mu(b)} \geq C^{-1} \frac{\mu(T_f^{|b|}b \cap _{-n_{j+1}}A_{-n_j, 0})}{\mu(T_f^{|b|}b)} \geq C^{-1}\gamma.
\end{equation}
Letting $\delta := C^{-1}\gamma > 0$ establishes (\ref{ineqn:decayb}), thereby proving the claim. 
\end{proof}

\begin{rmk}
In the special case where $T$ is full-branch and for all $g \in G$, $g(X) = \{-1, +1\}$, then the conclusion to Lemma \ref{lem:NoDiv-inf} stills holds when the assumption that the DWDE has the Linkage Property is replaced with the weaker assumption that it is irreducible.
\end{rmk}

\begin{lem}\label{lem:PD+PE+}
If the DWDE is irreducible with uniformly bounded jumps and a stationary environment, then $P(D^+) = 1$ implies that $P(E^+) > 0$.
\end{lem}
\begin{proof} 
By Lemma \ref{pro:DecompD+}(b), $P(D^+) = 1$ implies that for some $c \geq 1$, 
\begin{equation}\label{eqn:Aux1}
P(\{\omega: \mu(\cap_{k \geq 1}~{(_{0}}A_{0,k}(\omega))) \geq c^{-1}\}) \geq c^{-1}.
\end{equation}
For $k, c \geq 1$, define the sets $D(k, c) := \{\omega: \mu(_{0}A_{0,k}(\omega)) \geq c^{-1}\}$. It follows immediately from (\ref{eqn:Aux1}) that if $P(D^+) = 1$ then there exists $c \geq 1$ such that for all $k \geq 1$
\begin{equation}\label{eqn:PcapD}
P(D(k, c)) \geq c^{-1}.
\end{equation}
For $k, c \geq 1$ define the sets $E(k, c) := \{\omega: \mu(_{-k}A_{-k,0}(\omega)) \geq c^{-1} \}$. By Corollary \ref{cor:MeasR} the sets $D(k,c)$ are $P$-measurable. It follows from (\ref{eqn:PcapD}) and the $\eta$-invariance of $P$ that for all $k \geq 1$, 
\begin{equation}\label{eqn:PEkd}
P(E(k, c)) \geq c^{-1}.
\end{equation}
From (\ref{eqn:PEkd}), it is immediate that for all $k \in \N$
\[
P(\cup_{N \geq k} E(N, c)) \geq c^{-1}
\]
and that
\[
P(\cap_{k \geq 1} \cup_{N \geq k} E(N, c)) \geq c^{-1}.
\]
Since $E^+ \supset \cap_{k \geq 1} \cup_{N \geq k} E(N, c)$, it follows that $P(E^+) \geq c^{-1} > 0$.
\end{proof}

We have the following corollary of Theorem \ref{thm:nosplitcase}.
\begin{cor}\label{cor:symmetricRecurrence}
Suppose that the DWDE satisfies the hypotheses of Theorem \ref{thm:nosplitcase}, and that the deterministic environment is i.i.d.\ such that 
\begin{equation}\label{eq:Symmetry}
P(f_i = g) = P(f_i = -g),
\end{equation}
for all transition functions $g \in G$ and all $i \in \Z$. Then the DWDE is transitive on $\Z$.
\end{cor}
\begin{proof}
It follows from the symmetry of (\ref{eq:Symmetry}) that $P(R^+) = P(R^-)$. By Theorems \ref{thm:4Case0-1Law} and \ref{thm:nosplitcase}, it follows that $P(R^+) = P(R^-) = 1$, and hence the DWDE is transitive on $\Z$.
\end{proof}

\section{Transient DWDEs on $\Z$}\label{Sec:Transience}

In this section we establish hypotheses under which the DWDE is transient and, moreover, for which it is possible to determine whether case (i) or (ii) of Theorem \ref{thm:nosplitcase} holds. 

In addition to the standing hypotheses, we assume that $T$ is a Gibbs-Markov transformation, from which it automatically follows that $T$ has the Strong Distortion and Big Image properties. It can be shown (see \cite{AD2001}) that Gibbs-Markov maps have the property that there exists $C \geq 1$ such that for all $n \geq 1$, all $a \in \beta_n$ and all $x \in a$
\begin{equation}\label{eq:Gibbscylinder}
C^{-1}g(x)g(Tx)\cdots g(T^{n-1}x) \leq m(a) \leq Cg(x)g(Tx)\cdots g(T^{n-1}x),
\end{equation}
where $g := \frac{dm}{dm \circ T}$. The measure $m$ is known as a \emph{Gibbs} measure for the transformation $T$, and the function $-\ln g$ is called the \emph{Gibbs potential} of the measure $m$.

\begin{thm}\label{thm:TransientDWDE}
Suppose that the DWDE has an ergodic, stationary environment, and is driven by a full-branch Gibbs-Markov transformation $T$ of a probability space $(X, m)$, with finite Markov partition $\beta$ and Gibbs potential $h$. Suppose further that for some $1 \leq r \leq \#\beta - 1$, and all transition functions $f \in G$, 
\[
\#\{a \in \beta: f(a) = +1\} = r ~\quad \&~ \quad \#\{a \in \beta: f(a) = -1\} = \#\beta - r.
\]
Then:
\begin{itemize}
\item[(i)] If $\inf h > \frac{1}{2} \ln 4r(\#\beta - r)$ and $r > \frac{\#\beta}{2}$ then
\[
\lim_{n \rightarrow \infty} U_{i,n}(x, \omega) = +\infty, \mathrm{for~all}~i \in \Z,~ ~P-a.e.\ \omega ~\& ~ m-a.e.\ x. 
\]
\item[(ii)] If $\inf h > \frac{1}{2} \ln 4r(\#\beta - r)$ and $r < \frac{\#\beta}{2}$ then
\[
\lim_{n \rightarrow \infty} U_{i,n}(x, \omega) = -\infty, \mathrm{for~all}~i \in \Z,~ ~P-a.e.\ \omega ~\& ~ m-a.e.\ x.
\]
\end{itemize}
\end{thm}
The hypotheses of Theorem \ref{thm:TransientDWDE} are a special case of those of Theorem \ref{thm:nosplitcase}. In particular, as observed in Example \ref{ex:Linkage}, since $T$ is full-branch with a finite Markov partition $\beta$, and since $f(X) = \{+1, -1\}$ for all transition functions $f \in G$, the DWDE has the Linkage Property.

The proof of the Theorem \ref{thm:TransientDWDE} requires the following lemma, which derives from Theorem 4.4.3 in \cite{Aaronson1997}. (The routine details of this derivation can be found in \cite[Corollary 8.3]{Little2012a}.)

\begin{lem}\label{cor:Aaro}
Let $S$ be an irreducible Markov transformation of a probability space $(Y, q)$ that has the Strong Distortion Property. Then for any positive measure set $A$, if $S$ is recurrent (in the sense of Definition \ref{def:MarkovRecTrans}) then
\begin{equation}\label{eq:divcons1}
\sum_{j=0}^{\infty}q(A \cap S^{-j}A) = \infty,
\end{equation}
whereas if $S$ is transient (in the sense of Definition \ref{def:MarkovRecTrans}) then
\begin{equation}\label{eq:divdiss1}
\sum_{j=0}^{\infty}q(A \cap S^{-j}A) < \infty.
\end{equation}
\end{lem}

\begin{defn}
Fix $\theta \in (0,1)$. Let $\nu$ denote the weighted product measure on $X \times \Z$ such that for all $m$-measurable sets $A$ and all $i \in \Z$,
\begin{equation}\label{weightedmes}
\nu(A \times \{i\}) = \theta^{|i|}m(A)\cdot(\sum_{i=-\infty}^{\infty}\theta^{|i|})^{-1}. 
\end{equation} 
\end{defn}
Clearly, $\nu$ is a probability measure on $X \times \Z$. Moreover, we have the following proposition, the routine details of whose proof can be found in \cite[Proposition 8.4]{Little2012a}.
\begin{pro}\label{pro:NuDist}
For all $\omega \in \Omega$, the skew-product transformation $T_{f(\omega)}$ of the measure space $(X \times \Z, \nu)$ has the Strong Distortion Property.
\end{pro}

\begin{proof}{\textbf{of Theorem \ref{thm:TransientDWDE}}} We prove part (i) of the theorem as the proof of part (ii) is similar. In particular, we fix an arbitrary environment $(f_i)_{i \in \Z}$ satisfying the hypotheses of Theorem \ref{thm:TransientDWDE} and show that (a) the skew product transformation $T_{f}: (X \times \Z, \nu) \rightarrow (X \times \Z, \nu)$ is transient, and (b) that 
\begin{equation}\label{eq:Lim0}
\lim_{k \rightarrow \infty}m(_0A_{0,-k}) = 0,
\end{equation}
where for $k \geq 1$ define
\[
_0A^n_{0,-k} := \{x \in X: U_{0,n}(x) = -k,~\mathrm{and}~-k < U_{0,j}(x) < 0,~\mathrm{for}~1 \leq j \leq n-1\},
\]
and define $_0A_{0,-k} := \cup_{n=1}^{\infty}~ {_0}A^n_{0,-k}$. The result then follows from Theorem \ref{thm:nosplitcase}.

We first show that the deterministic walk in the environment $(f_i)_{i \in \Z}$ is transient. Fix a partition element $a \times \{0\} \in \tilde{\beta}$. Since $\nu$ is a finite measure on $X \times \Z$ and $\nu$ is equivalent to $\mu$, by Lemma \ref{cor:Aaro} and Proposition \ref{pro:NuDist}, it suffices to show that
\begin{equation}\label{eq:convergence}
\sum_{n=0}^{\infty} \nu(a \times \{0\} \cap T_f^{-n}(a \times \{0\})) < \infty.
\end{equation}
Since the deterministic walk is restricted to jumps of +1 and -1, it can only return to the partition element $a \times \{0\}$ in an even number of steps. From (\ref{weightedmes}) it follows that for all $2n+1$-cylinders that start and end in $a \times \{0\}$ we have
\begin{equation}\label{ineq:nu<mu}
\nu([a \times \{0\}, a_1 \times \{k_1\}, \ldots, a_{2n-1} \times \{k_{2n-1}\}, a \times \{0\}]) \leq
 m([a, a_1, \ldots, a_{2n-1}, a]).
\end{equation}
Given that we return to $a \times \{0\}$ at time $2n$, for each $n \geq 1$, there are $\left(\begin{matrix}2n\\n
\end{matrix}\right)$ numbers of ways choosing the timing of $n$ leftward and $n$ rightward jumps. Since $T$ is full branch, and since every transition function takes the value $+1$ on exactly $r$ elements of $\beta$, and the value $-1$ on the remaining $\#\beta - r$ elements of $\beta$, it follows that there are $r^{n}(\#\beta - r)^n\textstyle\left(\begin{matrix}2n\\n
\end{matrix}\right)$ cylinder sets of rank $2n+1$ that start and end in the partition element $a \times \{0\}$. 

Recalling the Radon-Nikodym derivative $g = \frac{dm}{d(m \circ T)}$, and defining 
\[
M := \sup g = e^{- \inf h}
\]
it follows from (\ref{eq:Gibbscylinder}) that there exists $C > 0$ such that for all $n \geq 1$,
\begin{equation}\label{eq:SupnCyl}
\sup_{a \in \beta_n} m(a) \leq CM^n.
\end{equation}
It follows from (\ref{ineq:nu<mu}) and (\ref{eq:SupnCyl})
\begin{equation}\label{eq:1stApp}
\sum_{n=0}^{\infty} \nu(a \times \{0\} \cap T_f^{-n}(a \times \{0\})) \leq C\sum_{n=0}^{\infty} r^{n}(\#\beta - r)^n\textstyle\left(\begin{matrix}2n\\n
\end{matrix}\right) M^{2n+1}.
\end{equation}
By Stirling's formula it follows that for large $n$
\begin{equation}\label{eq:Stirling}
\textstyle\left(\begin{matrix}2n\\n
\end{matrix}\right) \approx (2\pi n	)^{-\frac{1}{2}}4^n.
\end{equation}
It follows from (\ref{eq:1stApp}) and (\ref{eq:Stirling}) that for some constant $C' > 0$
\begin{equation}\label{eq:Est}
\sum_{n=0}^{\infty} \nu(a \times \{0\} \cap T_f^{-n}(a \times \{0\})) \leq C'	\sum_{n=0}^{\infty} (4r(\#\beta - r)M^2)^n.
\end{equation}
It follows from the hypotheses that $M^2 = e^{- 2 \inf h} < (4r(\#\beta - r))^{-1}$ and so 
\begin{equation}\label{eq:EstM}
4r(\#\beta - r)M^2 < 1.
\end{equation}
From (\ref{eq:EstM}) it follows that the right hand side of (\ref{eq:Est}) converges as required. This establishes that the deterministic walk is transient.

We now establish (\ref{eq:Lim0}). Given that the deterministic walk can only take jumps of +1 and -1, let $c_{n,k}$ denote the number of ways of choosing the timing $n+k$ leftward jumps and $n$ rightward jumps so that the first time the walk visits state $-k$ is at time $2n+k$, and that at no time before does it return to state $0$. It follows from the hypotheses that there are $c_{n,k}r^n(\#\beta - r)^{n+k}$ cylinder sets $a \in \beta_{2n+k}$ such that $a \subset ~_0A^{2n+k}_{0,-k}$. From (\ref{eq:SupnCyl}) we obtain 
\begin{gather}
\begin{aligned}\label{0A0-k_Est}
m({_0}A_{0,-k}) & \leq  C\sum_{n=0}^{\infty} c_{n,k}r^n(\#\beta - r)^{n+k}M^{2n+k}\\
                & =  C(\#\beta - r)^kM^k\sum_{n=0}^{\infty} c_{n,k}(r(\#\beta - r)M^2)^n.
\end{aligned}
\end{gather}
It follows from (\ref{eq:EstM}) that
\begin{equation}\label{eq:2R-1_b}
\textstyle r(\#\beta - r)M^2 < \frac{1}{4} 
\end{equation}
and since $r > \#\beta - r$ we also have that
\begin{equation}\label{eq:R-1_b}
\textstyle  (\#\beta - r)M < \frac{1}{2}.
\end{equation}
From (\ref{eq:2R-1_b}) and (\ref{eq:R-1_b}), there exists a positive integer $R$ such that 
\begin{equation}\label{eq:2R-1}
\textstyle r(\#\beta - r)M^2 < \frac{R(R-1)}{(2R-1)^2} <  \frac{1}{4} 
\end{equation}
and
\begin{equation}\label{eq:R-1}
\textstyle  (\#\beta - r)M < \frac{R-1}{2R-1} < \frac{1}{2}.
\end{equation}
It follows from (\ref{0A0-k_Est}), (\ref{eq:2R-1}), (\ref{eq:R-1}) that
\begin{equation}\label{LimZero_1}
m(_0A_{0,-k}) \leq C \left(\frac{R-1}{2R-1}\right)^k\sum_{n=0}^{\infty}c_{n,k} \left(\frac{R(R-1)}{(2R-1)^2}\right)^n.
\end{equation}
Consider the deterministic walk on $\Z$ defined by 
\begin{equation}\label{eq:HomogSRW}
V_n(x) := \sum_{j=0}^{n-1}f \circ T^j(x)
\end{equation}
where the transformation $T: ([0,1], \lambda) \rightarrow ([0,1], \lambda)$ is defined by
\[
Tx := (2R-1)x ~(\mathrm{mod}~1),
\]
and 
\begin{displaymath}
f(x) := \left\{ \begin{array}{ll} 
+1 & \mathrm{if}~x \in [0, \frac{R}{2R-1})\\ 
-1 & \mathrm{if}~x \in [\frac{R}{2R-1}, 1].
\end{array} \right.
\end{displaymath}
Clearly, $T$ is a piecewise linear and full-branch Markov map, with Markov partition $\mathcal{P}$, consisting of $2R-1$ intervals of equal length. Also the transition function $f$ satisfies
\begin{equation}\label{gProp}
\#\{a \in \mathcal{P}: f(a) = +1\} = R \quad \mathrm{and} \quad \#\{a \in \mathcal{P}: f(a) = -1\} = R-1.
\end{equation}
It is straightforward to show that the deterministic walk $V_n$ satisfying (\ref{eq:HomogSRW}) and (\ref{gProp}) is a model of a simple random walk whose probability of a leftward jump at any given time is $\frac{R-1}{2R-1}$, and whose probability of a rightward jump at any given time is $\frac{R}{2R-1}$. (For the routine details see \cite[Proposition 3.5]{Little2012a}.) It is immediate that this process diverges to the right with probability 1, and hence that 
\begin{equation}\label{eq:LimAk0}
\lim_{k \rightarrow \infty}\lambda({_0}B_{0,-k}) = 0,
\end{equation}
where $\lambda$ denotes Lebesgue measure and
\[
{_0}B_{0,-k} := \{x \in [0,1]: V_n(x) = -k,~\mathrm{and}~-k < V_j(x) < 0,~\mathrm{for}~1 \leq j \leq n-1\}.
\]
For all $n \geq 1$ and all $a \in \mathcal{P}_n$, $\lambda(a) = \left(\frac{1}{2R-1}\right)^n$, and it follows that
\begin{equation}\label{LimZero}
\lambda({_0}B_{0,-k}) = \left(\frac{R-1}{2R-1}\right)^k\sum_{n=0}^{\infty}c_{n,k} \left(\frac{R(R-1)}{(2R-1)^2}\right)^n. 
\end{equation}
Equation (\ref{eq:Lim0}) now follows from (\ref{LimZero_1}), (\ref{eq:LimAk0}) and (\ref{LimZero}). This completes the proof.
\end{proof}

\begin{examp}
Consider a DWDE for which the base transformation is a full-branch Gibbs-Markov map $T: ([0,1], \lambda) \rightarrow ([0,1], \lambda)$, with Markov partition $\beta := \{\textstyle[0,\frac{1}{3}), [\frac{1}{3}, \frac{2}{3}), [\frac{2}{3}, 1]\}$, such that $\lambda$ has Gibbs potential $h: [0,1] \rightarrow \R$ and $\inf h > \frac{1}{2}\ln 8$. Suppose the deterministic environment is ergodic and stationary and that the transition functions $f \in G$ satisfy $\#\{a \in \beta: f(a) = +1\} = 2$ and $\#\{a \in \beta: f(a) = -1\} = 1$. It follows from Theorem \ref{thm:TransientDWDE} that $\lim_{n \rightarrow \infty} U_{i,n}(x, \omega) = +\infty$ for all $i \in \Z$, $P$-a.e.\ $\omega$ and $m$-a.e.\ $x$. It is clear that the fastest rate at which measures of cylinders can decay is $\frac{1}{3}$. Theorem \ref{thm:TransientDWDE} says that as long as the measures of cylinders decay at a rate that is faster than $\frac{1}{\sqrt{8}}$ then the DWDE will still diverge to the right.
\end{examp}

\paragraph{Acknowledgements} 
I am greatly indebted to Ian Melbourne for suggesting \emph{deterministically driven random walks in a random environment} as a topic of research, and for his advice during the development this work. I would also like to thank the EPSRC for funding this research.

\end{document}